\documentclass[a4paper,12pt]{article}
\usepackage[english]{babel}
\usepackage{paralist,url,verbatim, anysize}
\usepackage{amscd}
\usepackage[all,cmtip]{xy}
\usepackage{amsmath,amsthm,amssymb,amsfonts}
\usepackage[bookmarksopen=true]{hyperref}
\usepackage[usenames, dvipsnames]{color}
\usepackage{graphicx}
\usepackage{hyperref}

\makeatletter
\let\@fnsymbol\@arabic
\makeatother

\theoremstyle{plain}
\newtheorem{theorem}{\bf Theorem}[section]

\newtheorem{corollary}[theorem]{Corollary}
\newtheorem{lemma}[theorem]{Lemma}
\newtheorem{proposition}[theorem]{Proposition}

\newtheorem{question}[theorem]{Question}

\theoremstyle{definition}

\newtheorem{definition}[theorem]{Definition}

\newtheorem{example}{\bf Example}

\newtheorem{thmnonumber}{\bf Theorem}
\newtheorem{cornonumber}[thmnonumber]{\bf Corollary}

\newcommand{\Z}{\mathbb{Z}}
\newcommand{\PP}{\mathbb{P}}

\newcommand{\CC}{\mathbb{C}}

\newcommand{\Tor}{\operatorname{Tor} }

\newcommand{\Proj}{\operatorname{Proj} }

\newcommand{\pp}{\mathfrak{p}}

\newcommand{\mm}{\mathfrak{m}}
\newcommand{\Lyu}{\mathfrak{L}}
\renewcommand{\O}{\mathcal{O}}

\newcommand{\reg}{\operatorname{reg} }
\definecolor{mypink}{RGB}{215, 5, 234}

\begin{document}
\author{
Bruno Benedetti \thanks{Supported by NSF Grant 1600741, `Geometric Combinatorics and Discrete Morse Theory'} \\
\small Dept. of Mathematics\\ \small University of Miami\\
\footnotesize \url{bruno@math.miami.edu}
\and
Michela Di Marca \thanks{Supported by PRIN 2010S47ARA 003, `Geometria delle Variet\`a Algebriche.'} \\
\small Dip. di Matematica\\ \small Universit\`{a}  di Genova\\
\footnotesize \url{dimarca@dima.unige.it}
\and
Matteo Varbaro \thanks{Supported by PRIN 2010S47ARA 003, `Geometria delle Variet\`a Algebriche.'} \\
\small Dip. di Matematica\\ \small Universit\`{a}  di Genova\\
\footnotesize \url{varbaro@dima.unige.it}
}
\title{Regularity of line configurations}
 \date{October 2, 2017}
\maketitle

\begin{flushright} \em \footnotesize
Dedicated to the memory of Tony Geramita \\ \vskip4mm
\end{flushright}

\begin{abstract}
We show that in arithmetically-Gorenstein line arrangements with only planar singularities, each line intersects the same number of other lines. This number has an algebraic interpretation: it is the Castelnuovo--Mumford regularity of the coordinate ring of the arrangement. %

We also prove that every $(d-1)$-dimensional simplicial complex whose 0-th and 1-st homologies are trivial is the nerve complex of a suitable $d$-dimensional standard graded algebra of depth $\ge 3$. This provides the converse of a recent result by Katzman, Lyubeznik and Zhang.
\end{abstract}

\section*{Introduction}
The study of lines on smooth surfaces of $\PP^3$ has a fascinating history. Lines on smooth cubics were investigated in the Nineteenth century by Cayley \cite{Ca}, Salmon \cite{Sa}, and Clebsch \cite{Cl}, among others. Every smooth cubic contains exactly 27 lines, whose pattern of intersection is independent of the chosen cubic. In 1910 Schoute showed that the 27 lines can be put into a one-to-one correspondence with the vertices of a $6$-dimensional polytope, so that all incidence relations between the lines are mirrored in the combinatorial structure of the polytope \cite{Schoute} \cite{DuVal}. 

The situation changes drastically for surfaces of larger degree. In fact, the generic surface of degree $d \ge 4$ contains no line at all. However, special smooth surfaces do contain lines (they are forced to be in a finite number whenever $d\geq 3$ though). An example of a smooth quartic containing 64 lines was found in 1882 by Schur \cite{Schur}:
\[x_0^4-x_0x_1^3=x_2^4-x_2x_3^3.\]
In 1943 B.~Segre \cite{Se} claimed that a smooth quartic may contain at most 64 lines, which is precisely the number achieved by Schur's quartic. Segre's proof, however, was based on the erroneous assumption that a line on a smooth quartic can meet at most 18 other lines on it. In 2015 the work of Rams and Sch\"utt \cite{RS} exhibited smooth quartics in which some line actually meets 20 other lines. Using a deeper argument, however, Rams and Sch\"utt  were able to salvage Segre's conclusion that indeed no smooth quartic contains more lines than Schur's.

We are naturally intrigued by the symmetry that many of these line configurations on smooth surfaces seem to enjoy. For example: 
\begin{compactitem}
\item[(i)] each line on a smooth cubic meets exactly $10$ of the others;
\item[(ii)] each line on the Schur quartic meets exactly $18$ of the others;
\item[(iii)] each line on the degree-$d$ Fermat surface meets exactly $4d-2$ of the others.
\end{compactitem} 
Where does this regularity come from? Schoute's ``polytopal bijection'' offers an explanation only for (i). The main result of the present paper provides a general, two-line answer to this question.

\begin{thmnonumber}[Theorem \ref{thm:Main}] \em \label{thm:MainThm}
In any arithmetically-Gorenstein line arrangement $X \subseteq \PP^n$ where all singularities are planar, each line intersects exactly $\reg X -1$ of the other lines. 
\end{thmnonumber}

The planarity of all singularities is necessary: Without it, we can only claim that each line intersects \emph{at least} $\reg X -1$ of the other lines, cf.~\cite[Theorem 3.8]{BV}. But curves on a smooth surface have by definition only planar singularities. In particular, if a line arrangement $X \subseteq \PP^3$ is the complete intersection of a smooth surface of degree $d$ and another surface (not necessarily smoooth) of degree $e$, Theorem \ref{thm:MainThm} reveals that each line of the arrangement must intersect exactly $d + e -2$ of the other lines. This, as we will see in the subsection on line arrangements in $\PP^3$, explains cases (i), (ii) and (iii) above.
 
From Theorem \ref{thm:MainThm} one can easily infer the following:

\begin{cornonumber}[Corollary~\ref{cor:VeryAmple}]
Let $X$ be a smooth surface, $H$ a very ample divisor, $X\subseteq \PP^n$ the embedding given by the complete linear system $|H|$. Let $D_1,\ldots ,D_s$ be lines on $X\subseteq \PP^n$ such that $D_1+\ldots +D_s \sim dH$ for some $d\in \mathbb{Z}_{>0}$.
\begin{compactitem}
\item[{\rm (i)}] If $n=3$, then $|\{j\neq i:D_j\cap D_i\neq \emptyset\}|=\deg H+d-2  $ for each $ i=1,\ldots , s$.
\item[{\rm (ii)}] If $X$ is a $K3$ surface, then $|\{j\neq i:D_j\cap D_i\neq \emptyset\}|=d+2 $ for each $i=1,\ldots , s$.
\end{compactitem}
\end{cornonumber}

Theorem \ref{thm:MainThm} can also be viewed as a step forward in the problem of computing the Castelnuovo-Mumford regularity of a subspace arrangement, studied among others by Derksen--Sidman \cite{DS}.

The second section of the present paper is devoted to the related problem of understanding the geometry of a simplicial complex that is attached to any standard graded $\mathbb{C}$-algebra $R$, called the {\it Lyubeznik complex} or the {\it nerve complex} of $R$, and here denoted by $\mathfrak{L}(R)$. A recent result by Katzman, Lyubeznik and Zhang \cite{KLZ} states that when $\operatorname{depth } R \ge 3$ one has $\widetilde{H}_0(\Lyu(R);\mathbb{C})=\widetilde{H}_1(\Lyu(R);\mathbb{C})=0$. We show the converse:

\begin{thmnonumber}[Theorem \ref{thm:Lyubeznik}] \em
For any $(d-1)$-dimensional simplicial complex $\Delta$ with $\widetilde{H}_0(\Delta;\mathbb{C})=\widetilde{H}_1(\Delta;\mathbb{C})=0$, there exists a $d$-dimensional standard graded $\mathbb{C}$-algebra $R$ such that $\operatorname{depth } R \ge 3$ and $\Lyu(R)=\Delta$. 
\end{thmnonumber}

Our construction has the advantage of being compatible with the dual graph notion, in the sense that the dual graph of $R$ is simply the $1$-skeleton of $\Delta$.

\subsection*{Notation}
Let $K$ be a field, and $X$ a projective scheme over $K$. Fix a closed embedding $X\subseteq \PP^n$, and let $I_X\subseteq K[x_0,\ldots ,x_n]=:S$ the corresponding saturated ideal. Let $R_X:=S/I_X$ be the corresponding coordinate ring. We say that $X\subseteq \PP^n$ is {\it arithmetically-Gorenstein} if $R_X$ is Gorenstein. Furthermore, the {\it Castelnuovo-Mumford regularity} of $X\subseteq \PP^n$ is 
\[\reg X:=\max\{j-i:\Tor_i^S(I_X,K)_j\neq 0\}.\]
We say that $X\subseteq \PP^n$ is a {\it subspace arrangement} if $I_X=\bigcap_{i=1}^s\pp_i$ where the $\pp_i$'s are generated by linear forms. We say that $X\subseteq \PP^n$ is a {\it line arrangement} if it is a subspace arrangement and $\dim S/\pp_i=2$ for all $i=1,\ldots ,s$.

\enlargethispage{2mm}
By the {\it dual graph} of $X$, denoted by $G_X$, we mean the simple graph whose vertices are the $s$ irreducible components  $X_1,\ldots ,X_s$ and whose edges are
\[\{\{i,j\}:\dim X_i\cap X_j = \dim X-1\}\]
The {\em valency} of a vertex is the number of vertices adjacent to it. (This is usually called ``degree'' in graph theory, but we refrain from this notation to avoid confusion.)  We denote by $\delta(G_X)$ the minimum valency of a vertex. By $\Delta(G_X)$ we denote the maximum valency. If all vertices have the same valency (say, $k$), we say that the graph $G_X$ {\em is regular} (or \emph{has regularity $k$}). 
The graph $G_X$ is {\em $k$-connected} if it has at least $k+1$ vertices and the deletion of less than $k$ vertices from $G_X$, however chosen, yields a connected graph. 
It is easy to see that if $G_X$ is $k$-connected, then  $\delta(G_X) \ge k$. The \emph{distance} of two vertices is the smallest number of edges of a path connecting them. The maximum distance in a graph is called {\em (graph) diameter}. It is well known that any $k$-connected graph with $n$ vertices has diameter at most $\lfloor \frac{n+k-2}{k}\rfloor$.

\section{The two notions of regularity coincide}
Let us start by showing that the dual graph of an arithmetically-Gorenstein line arrangement need not be regular. Recall that by \cite[Theorem 3.8]{BV} we have the inequality 
\[\reg X -1  \le \delta (G_X)  \le \Delta (G_X).\]

\begin{example} \label{prp:ex1} 

Let $I \subset \mathbb{Q}[x,y,z,t,w]$ be the following ideal:
\[ I = (x^2-y^2+z^2-t^2,\: xz-yt, \: xw). \]
With the help of the software { Macaulay2} \cite{M2}, one can immediately verify that $I$ is the intersection of the following $8$ prime ideals.
\[ \!\!
\begin{array}{llll}
\pp_1=(t,y-z,x) & \:\pp_3=(z-t,y,x) &  \:\pp_5=(w,z-t,x-y) & \:\pp_7=(w,y+z,x+t)\\
 \pp_2=(t,y+z,x) & \:\pp_4=(z+t,y,x) &\:\pp_6=(w,z+t,x+y)  &\:\pp_8=(w,y-z,x-t)
\end{array}
\]
Hence, $I$ defines an arrangement $X\subset \PP^4$ of 8 lines. One can check that the dual graph is
\[G_X=  12, 13, 14, 18, 23, 24, 27, 34, 35, 46, 57, 58, 67, 68.
\]
While the first four vertices have valency $4$, the remaining four vertices have valency $3$. Being a complete intersection of three quadrics, $X$ has regularity 4. In conclusion, $\delta(G_X)=3$ and $\Delta(G_X)=\reg X=4$.
\end{example}

The graph of Example \ref{prp:ex1} is ``very close'' to being $3$-regular. In fact, it has minimum valency equal to $\reg X -1$ and maximum valency equal to $\reg X$. 
This triggers two questions: 
\begin{compactenum}[(1)]
\item Is perhaps $\delta (G_X)$ always equal to $\reg X -1$? 
\item Is there an upper bound on how large the gap $\Delta(G_X) -  \reg X$ can be?
\end{compactenum}
Our next step is to provide a negative answer to both questions. It turns out that even when the dual graph of $X$ is regular, the graph-theoretic regularity of $G_X$ may be {\em arbitrarily larger} then the Castelnuovo--Mumford regularity of $X$.

\begin{proposition} For any integers $n, d \ge 2$, there is a complete intersection $X\subseteq \PP^{n+1}$ of Castelnuovo--Mumford regularity $nd -n+1$ whose dual graph is  $(d^n-1)$-regular and $(d^n-1)$-connected.
\end{proposition}

\begin{proof}
Let $f_1,\ldots ,f_n$ be generic homogeneous polynomials 
of degree $d$ 
in $K[x_0,\ldots ,x_n]$. Let $Y\subseteq \PP^n$ be the projective scheme defined by the ideal $J=(f_1,\ldots ,f_n)$. Since the $f_i$'s are generic, the ideal $J$ is a radical complete intersection, so $Y\subseteq \PP^n$ consists of 
$d^n$
distinct points. Consider the cone $X\subseteq \PP^{n+1}$ of $Y$. This is an arrangement of 
$d^n$ 
lines in $\PP^{n+1}$. Being a complete intersection, $X$ is arithmetically-Gorenstein and 
$\reg X=nd -n+1$.
Since all lines in $X$ pass through $[0,\ldots ,0,1]\in\PP^{n+1}$, the dual graph of $X$ is the complete graph on  
$d^n$ vertices.
\end{proof}

We are now ready for our main result. 

\begin{definition}[Planar singularities] A curve $(X,\mathcal{O}_X)$ over $K$ has \textit{only planar singularities} if \[\widehat{\mathcal{O}_{X,P}}\cong K[[x,y]]/I_P \;\:\textrm{ for all } P \in X \ \textrm{ and for some }I_P\subseteq K[[x,y]].\] 
\end{definition}

For line arrangements in $\mathbb{P}^n$, this boils down to the following condition: whenever three or more lines come together at a single point, all those lines must belong to the same plane. Of course, every line arrangement where \emph{no} three lines meet at a common point, yields an example of a curve with only planar singularities. 
By definition, any curve lying on a smooth surface has only planar singularities.

\begin{theorem} \label{thm:Main}
Let $X \subseteq \PP^n$ be an arithmetically-Gorenstein line arrangement with only planar singularities. The dual graph $G_X$ is $(\reg X -1)$-regular and $(\reg X -1)$-connected.
\end{theorem}
\begin{proof} Choose an arbitrary ordering $L_1, \ldots, L_s$ of the lines of $X$. Let $d$ be the number of lines intersecting the last line $L_s$. We are going to show that $d=\reg X -1$, whence the conclusion follows by reshuffling the order. 

Let $I=\bigcap_{i=1}^s\pp_i\subseteq S=K[x_0,\ldots ,x_{n}]$ be the defining ideal of $X\subseteq \PP^n$. Set $J=\bigcap_{i=1}^{s-1}\pp_i$. 
The ideal $J+\pp_s$ defines the scheme $L_s\cap (L_1\cup L_2 \cup \ldots \cup L_{s-1})$, which is possibly not reduced, and whose underlying topological space consists of $k\leq d$ points of $L_s$.
Hence we can find a polynomial $g\in S$ of degree $h$, where $k\leq h \leq d$, such that
\[J+\pp_s=\pp_s+(g).\]
Note that $h=d$: in fact, $g=\prod_{i=1}^kq_i^{a_i}$, where $q_i$ is the equation of $P_i$ in $L_s$; since $P_i$ is a planar singularity of $X$, the number $a_i$ counts the lines of $X$ that are different from $L_s$ and pass through~$P_i$.
In particular, $\Tor_n^S(S/(J+\pp_s),K)_{n-1+d}\neq 0$. 

But  $\Proj(S/J)$ is geometrically linked to $L_s$ by $X$; it follows that $S/J$ is a 2-dimensional Cohen--Macaulay ring, and in particular 
\begin{equation} \label{eq:1}
\Tor_n^S(S/J,K)=\Tor_n^S(S/\pp_s,K)=0.
\end{equation} 
 Now, consider the short exact sequence
\[0\rightarrow S/I\rightarrow S/J\oplus S/\pp_s\rightarrow S/(J+\pp_s)\rightarrow 0.\]
Applying Tor and using equation (\ref{eq:1}), we see that there is an injection
\[\Tor_n^S(S/(J+\pp_s),K)_{n-1+d} \: \hookrightarrow \: \Tor_{n-1}^S(S/I,K)_{n-1+d}.\]
So also $\Tor_{n-1}^S(S/I,K)_{n-1+d} \ne 0$. By definition of Castelnuovo--Mumford regularity, we obtain
\[\reg(S/I) \ge d.\]
On the other hand, the graph $G(X)$ is $\reg(S/I)$-connected by \cite[Theorem 3.8]{BV}. In any $k$-connected graph, no vertex can have less than $k$ neighbors. 
This means that $\reg(S/I) \le d$. Hence $d=\reg(S/I) = \reg X - 1$. 
\end{proof}

\begin{corollary} \label{cor:VeryAmple}
Let $X$ be a smooth surface, $H$ a very ample divisor, $X\subseteq \PP^n$ the embedding given by the complete linear system $|H|$, and $D_1,\ldots ,D_s$ lines on $X\subseteq \PP^n$ such that $D_1+\ldots +D_s \sim dH$ for some $d\in \mathbb{Z}_{>0}$: 
\begin{compactitem}
\item[{\rm (i)}] If $n=3$, then $|\{j\neq i:D_j\cap D_i\neq \emptyset\}|=\deg H+d-2  $ for each $ i=1,\ldots , s$.
\item[{\rm (ii)}] If $X$ is a $K_3$ surface, then $|\{j\neq i:D_j\cap D_i\neq \emptyset\}|=d+2 $ for each $i=1,\ldots , s$.
\end{compactitem}
\end{corollary}

\begin{proof}
Consider the line arrangement $Y:=\bigcup_{i=1}^sD_i\subseteq \PP^n$. Being contained in a smooth surface, it has only planar singularities. If $n=3$, then $Y\subseteq \PP^3$ is the scheme-theoretic intersection of $X$ and some surface of degree $d$. In particular, it is arithmetically-Gorenstein of Castelnuovo--Mumford regularity $\deg H+d-2$. This settles part (i). 

For (ii): By definition the embedding $X\subseteq \PP^n$ is linearly normal. So it is projectively normal by results of Saint-Donat \cite{SD}. 
We claim that $H^1(X,\O_X(k))=0$ for all $k\in\Z$. Since $K_X\sim 0$ and $H^1(X,\O_X)=0$, in characteristic 0 this claim would immediately follow by the Kodaira vanishing and Serre's duality. However one can easily argue in any characteristic, by considering the following short exact sequence:
\[0\rightarrow \O_X(k-1)\rightarrow \O_X(k)\rightarrow \O_H(k)\rightarrow 0.\]
Because $H^0(X,\O_X(k))\cong H^0(H,\O_H(k))$ for all $k\leq 0$ and $H^0(X,\O_X(k))=0$ for all $k<0$, the map
\[H^1(X,\O_X(k-1))\rightarrow H^1(X,\O_X(k))\]
is injective for all $k\leq 0$. Since $H^1(X,\O_X)=0$ we therefore infer by induction that $H^1(X,\O_X(k))=0$ for all $k\leq 0$. Since $K_X\sim 0$, by Serre's duality we obtain that $H^1(X,\O_X(k))=0$ for all $k\in\Z$, and the claim is proven. Hence, the embedding $X\subseteq \PP^n$ is arithmetically Cohen-Macaulay. Yet $K_X\sim 0$, so the embedding is also arithmetically-Gorenstein of (Castelnuovo--Mumford) regularity 4. Since $Y\subseteq \PP^n$ is scheme-theoretically the intersection of $X$ with a hypersurface of degree $d$ in $\PP^n$, it is arithmetically-Gorenstein of regularity $d+3$.
\end{proof}

\subsection*{Applications to line arrangements in $\mathbb{P}^3$}
Having codimension $2$, for curves in $\mathbb{P}^3$ the properties of being a complete intersection and that of being arithmetically-Gorenstein are the same. Before digging into examples, we would like to clarify in what sense it is restrictive for a line arrangement to live in a three-dimensional ambient space. 

\begin{lemma}\label{lem:diameter}
Let $X\subseteq \PP^3$ be a line arrangement. If $X$ is the complete intersection of two surfaces of degree $d$ and $e$ respectively, then the diameter of $G_X$ is at most $\min\{d,e\}$.  
\end{lemma}

\begin{proof}
By \cite[Theorem 3.8]{BV}, $G_X$ is a $(d+e-2)$-connected graph on $de$ vertices. Hence the diameter is at most
$\lfloor \frac{de + (d+e-2) - 2}{d+e-2}\rfloor$, a quantity never larger than $\min \{d,e \}$.
\end{proof}

\begin{proposition} A graph $G$ is the dual graph of a line arrangement in $\mathbb{P}^n$, for some $n \ge 3$, if and only if it is the dual graph of some line arrangement in $\PP^3$. However, some graphs are dual to complete intersection line arrangements $Y\subseteq \PP^n$ for some $n \ge 4$, but cannot be realized as dual graph of any complete intersection line arrangement $X\subseteq \PP^3$.
\end{proposition}

\begin{proof}
Let $Y\subseteq \PP^n$ be a line arrangement such that $G=G_Y$. The secant variety $Z$ of $Y\subseteq \PP^n$ has dimension at most 3, so if $n \ge 4$ there is a point $P\in\PP^n\setminus Z$. Denoting by $\pi$ the projection from $P$ to some $\PP^3\subsetneq \PP^n$ not containing $P$, then $X:=\pi(Y)\subseteq \PP^3$ is a line arrangement isomorphic to $Y$. This settles the first part of the claim. However, the complete intersection property is not preserved under projections. In fact, consider
\[I=(a_1b_1, \ a_2b_2, \ a_3b_3)=\bigcap_{\sigma\subseteq \{1,2,3\}}(a_i, \ b_j:i\in \sigma,j\notin \sigma)\subseteq S=K[x_0,\ldots ,x_4],\]
where $a_i$ and $b_i$ are linear forms of $S$ such that $\dim_K\langle a_i, \ b_j:i\in \sigma,j\notin \sigma\rangle=3$ for all $\sigma\subseteq \{1,2,3\}$. Clearly, $I$ defines an arrangement of 8 lines in $\PP^4$ which is a complete intersection. Note that its dual graph has diameter $3$. Now, suppose by contradiction that some complete intersection line arrangement $X\subseteq \PP^3$ had the same dual graph. Then $X$ would consist of 8 lines. Let $f, g$ be the two homogeneous polynomials such that $(f,g)=I_X$. From the fact that $\deg f \cdot \deg g = 8$ we infer that $\min\{\deg(f),\deg(g)\}$ is either $1$ or $2$. Yet the diameter of $G_X$ is $3$, which contradicts Lemma \ref{lem:diameter}. 
\end{proof}

We are now ready to discuss a few famous examples of line arrangements in $\mathbb{P}^3$. For some of them, the property of being a complete intersection is not at all obvious. For this reason and for didactical purposes, we examine each example in detail. 

\begin{example}(Lines From Two Rulings)
The complete bipartite graph 
$K_{m,n}$ is the dual graph of a line arrangement. This can be realized in any smooth quadric in $\mathbb{P}^3$, by picking $m$ lines from one of the two rulings of the quadric, and then by picking $n$ further lines from the other ruling. The resulting arrangement $A\subseteq \mathbb{P}^3$ was studied by Geramita--Weibel \cite{GW} and later by Teitler--Torrance \cite{TT}, who computed the Castelnuovo-Mumford regularity and classified the arithmetical Cohen--Macaulayness of $A$ in terms of $m$ and $n$. Assuming $m,n \ge 3$, $A$ is arithmetically Cohen--Macaulay if and only if the integers $m$ and $n$ differ by $0$ or $1$ \cite[Theorem 1.2]{TT}.

We claim that $A$ is a complete intersection if and only if $m=n$. The ``only if'' part follows from Theorem \ref{thm:Main} (or from a direct argument). For the ``if'' part: when $m=n$, the arrangement $A$ is the complete intersection of the quadric with a union of $n$ planes. Since any edge in the dual graph corresponds to a pair of coplanar lines, the choice of the $n$ planes correspond to the choice of a complete matching of the graph.
\end{example}

\begin{example}[Twenty-Seven Lines] \label{ex:27}
Any smooth cubic surface $Y$ in $\mathbb{P}^3$ contains exactly $27$ lines. Any such cubic $Y$ is isomorphic to the blow-up of the projective plane $\mathbb{P}^2$ along $6$ points $P_1,\ldots ,P_6$. The $27$ lines can be described as follows:
\begin{compactitem}
\item[(a)] the exceptional divisor $E_i$ corresponding to $P_i$, with $i \in \{1,\ldots ,6\}$ (for a total of 6 lines of this type);
\item[(b)] the strict transform $L_{ij}$ of the line in $\mathbb{P}^2$ that passes through $P_i$ and $P_j$, with $i,j \in \{1,\ldots ,6\}$ and $i <j$ (which yields a total of $15$ further lines);
\item[(c)] the strict transform $C_i$ of the unique conic in $\mathbb{P}^2$ that passes through all points except $P_i$, with $i \in \{1,\ldots ,6\}$ ($6$ further lines).
\end{compactitem} 
Let $B$ be the line arrangement given by the $6+15+6=27$ lines and let $G_B$ be the dual graph of $B$. Consistently with the above notation, we denote by $E_i, L_{ij}$ and $C_i$ the vertices of $G_B$. Then, by construction, $G_B$ consists of the following edges:
\begin{compactitem}
\item $\{E_i,L_{hk}\}$ if and only if $i=h$ or $i=k$;
\item $\{L_{ij},C_k\}$ if and only if $i=k$ or $j=k$;
\item $\{E_i,C_j\}$ if and only if $i \neq j$.
\item $\{L_{ij},L_{hk}\}$ if and only if $\{i,j\}\cap\{h,k\}=\emptyset$.
\end{compactitem}
It is straightforward to check that $G_B$ is a $10$-connected $10$-regular graph. Theorem \ref{thm:Main} confirms this fact. Indeed, $B$ has only planar singularities and is the complete intersection of the cubic with the union of $9$ planes. (Each triangle in $G_B$ corresponds to a plane, the one containing the three lines, and $G_B$ can be partitioned into $9$ triangles.)
\end{example}

\begin{example}[Steiner set] A Steiner set is a line arrangement given by $3$ sets of $3$ lines each, where each line is incident with $2$ from each of the other two sets. 
In the notation of  Example \ref{ex:27}, let $G_B$ be the dual graph of an arrangement $B$ of 27 lines on some smooth cubic $Y$. Let  $G_C$ be the subgraph of $G_B$ induced on the following vertices:
\[\{E_1,E_2,E_3\} 
\: \: \cup \: \: 
\{L_{ij} | \,1\leq i<j \leq 3\}
\: \: \cup \: \: 
\{C_1,C_2,C_3\}.\]
This $G_C$ is $4$-regular, it has $9$ vertices and  $18$ edges. The subarrangement $C$ of $B$ dual to $G_C$ is a Steiner set. 
We can partition $G_C$ in 3 triangles given by the triples of lines:
\[\{E_1,L_{12},C_2\}, \ \{E_2,L_{23},C_3\}, \ \{E_3,L_{13},C_1\}.\]
Each triple of lines lies on a plane, so the line arrangement $C$ lies on the union $Z$ of the three planes. So we have $C\subseteq Y\cap Z$, and for degree reasons the inclusion must be an equality. Thus $C\subseteq \PP^3$ is a complete intersection of the smooth cubic $Y$ and a union $Z$ of three planes.  Theorem \ref{thm:Main} claims in fact that $G_C$ is $(3+3-2)$-regular.
\end{example}

\begin{example}[Schl\"afli's double-six] \label{ex:Schlaefli}
Let $G_D$ be the bipartite graph on $\{a_1,\ldots ,a_6\}\cup\{b_1,\ldots ,b_6\}$ such that $\{a_i,b_j\}$ is an edge if and only if $i\neq j$. Clearly $G_D$ is a $5$-regular graph. Schl\"afli's double-six is a line arrangement $D\subseteq \mathbb{P}^3$ having $G_D$ as dual graph; it consists of 12 of the 27 lines on a smooth cubic $Y$, and precisely $E_1,\ldots,E_6, C_1,\ldots, C_6$, with the notation of Example \ref{ex:27}. 

Since $G_D$ is $5$-regular and triangle-free, by picking the intersection points of $D$ we get a set $S$ of $30$ distinct points such that any line of $D$ passes through exactly $5$ points of $S$.
Next, choose 4 non-coplanar points $x,y,v,w$ outside the cubic, and consider the set $S' = S \cup \{x,y,v,w\}$. Since the linear system of quartics of $\PP^3$ has dimension $34$, we can find a quartic $Z$ passing through all points of $S'$. Since $Z$ contains $5$ points for each line of $D$, then it must contain  $D$. In addition, $Z$ cannot contain the cubic $Y$ (otherwise $Z$ would be $Y$ union a plane, against the choice of $x,y,v,w$). Then $Y\cap Z$ is a complete intersection of degree $12$ containing $D$; and the equality $D=Y\cap Z$ follows because $\deg D=12$. This is in accordance with Theorem \ref{thm:Main}, which states that $G_D$ is $(4+3-2)$-regular.
\end{example}

Contrarily to the previous examples, the graph $G_D$ has diameter $3>2$. This is somehow unexpected to us, so we would like to pose the following:

\begin{question}
Given a positive integer $d$, is there always some complete intersection line arrangement $X\subseteq \PP^3$ such that $G_X$ has diameter $d$?
\end{question}

Notice that a line arrangement as above should be the complete intersection of two surfaces of degree $\ge d$, by Lemma \ref{lem:diameter}.

\paragraph*{Higher-degree surfaces.} As mentioned in the introduction, the generic smooth surface of degree $d\geq 4$ contains no line. Furthermore, by a result of B. Segre \cite{Se}, any smooth surface of degree $d\geq 4$ cannot contain more than $(d-2)(11d-6)$ lines. Examples of smooth surfaces in $\mathbb{P}^3$ with many lines are those of equation
\[F(x_0,x_1,x_2,x_3)=\phi(x_0,x_1)-\psi(x_2,x_3),\]
where $\phi$ and $\psi$ are two arbitrary homogenous polynomials of degree $d$. 
As shown in \cite{BS}, in this case the maximal number of lines $N_d$ is given, for $d \ge 3$, by the formula

\[N_d 
= \left \{
\begin{array}{ll}
64  & \textrm{if }  d=4, \\
180 & \textrm{if }  d=6,\\
256 & \textrm{if }  d=8,\\
864 & \textrm{if }  d=12,\\
1600 & \textrm{if }  d=20,\\
3d^2 &\textrm{otherwise.} 
\end{array}
\right.
\]

The maximum is almost always achieved by the degree-$d$ Fermat surfaces, which contains exactly $3d^2$ lines; for $d=4$ the maximum is achieved by the Schur's quartic. Below we examine these two examples more closely.

\begin{example}[Schur's quartic]
Schur's quartic $\mathfrak{S}$ \cite{Schur} 
is the complex surface defined by the polynomial \[F(x_0,x_1,x_2,x_3)=x_0^4-x_0x_1^3-x_2^4+x_2x_3^3.\]
This is a particular example of polynomial of the type $\phi(x_0,x_1)-\psi(x_2,x_3)$, with $\psi = \phi$. Schur's quartic contains an arrangement $E$ of $64$ lines. We claim that $E$ is a complete intersection. 
To see this, we first compute explicitly the lines of $E$, by looking at the automorphisms of $\PP^1$ that fix the four roots of $\phi$ (cf.~the proof of \cite[Thm.~3.1]{BS}); such a group looks like $A_4$.
Choose the $4$ even permutations
\[a=\mathrm{id}, \quad 
b=\left(
\begin{array}{l}
1\ 2\ 3\ 4\\
3\ 4\ 1\ 2
\end{array} \right),
\quad
c=\left(
\begin{array}{l}
1\ 2\ 3\ 4\\
4\ 3\ 2\ 1
\end{array} \right),
\quad
d=\left(
\begin{array}{l}
1\ 2\ 3\ 4\\
2\ 1\ 4\ 3
\end{array} \right).
\] 
To these $4$ permutations correspond $4$ quadrics that, up to reordering the roots of $\phi$, are defined~by
\[\begin{array}{l}
Q_a=x_0x_3-x_1x_2,\\
Q_b=x_0(2\zeta^2x_2+x_3)-x_1(-x_2+\zeta x_3),\\
Q_c=x_0(2x_2+x_3)-x_1(-x_2+x_3),\\
Q_d=x_0(-2\zeta^2x_2-\zeta x_3)-x_1(\zeta x_2-x_3),
\end{array}\] 
where $\zeta$ is a primitive $3$rd root of unity.

Each quadric contains $8$ lines of $\mathfrak{S}$. Moreover, no line of $\mathfrak{S}$ belongs to two of the above four quadrics since $|\{a(i),b(i),c(i),d(i)\}|=4$ for $i=1,\ldots,4$. Hence, the union of the $4$ quadrics contains exactly $32$ of the lines of $\mathfrak{S}$. Let us call $E_1$ the arrangement formed by these lines.
Being $F$ irreducible, the ideal $I_1=(F,\ Q_a  Q_b  Q_c Q_d)$ is a complete intersection; 
and since it contains $32$ lines, for degree reasons it follows that $I_1$ is itself the ideal of $E_1$.

It remains to find the ideal $I_2$ of the arrangement $E_2$ formed by the remaining $32$ lines. For this we have to consider the remaining $8$ automorphisms of $\PP^1$ that fix the roots of $\phi$, corresponding to the $8$ permutations in $A_4$ different from $a,b,c,d$. Each automorphism yields a quadric having $8$ lines in common with $\mathfrak{S}$: $4$ in the first ruling (already contained in $E_1$), and $4$ in the second ruling. The latter, ranging over the $8$ quadrics, are the remaining $32$ lines. With the help of {Macaulay2} \cite{M2}, we verified that $I_2$ is generated by $F$ and by the irreducible octic surface 
\[\begin{array}{ll}P=&80x_0^3x_1^5+x_1^8+64x_0^2x_1^2x_2^4+64x_0^3x_1x_2^3x_3
+8x_1^4x_2^3x_3+\\
&64x_2^6x_3^2-64x_0^2x_1^2x_2x_3^3+
8x_0^3x_1x_3^4+x_1^4x_3^4+16x_2^3x_3^5+x_3^8 \end{array}
\]
Hence, $I_2$ is also a complete intersection. It follows that the ideal $I$ defining the whole arrangement $E$ of the $64$ lines is a complete intersection given by
\[I=(F,\; P \,Q_a \, Q_b \, Q_c \, Q_d).\]
Accordingly with Theorem \ref{thm:Main}, the dual graph of $E$ is $18$-regular. 
Moreover, $E$ can be partitioned into two subarrangements $E_1$ and $E_2$ that are themselves complete intersections; in particular, in $E$ each line of $E_1$ meets exactly $4+8-2=10$ lines of $E_1$ and $18-10=8$ lines of $E_2$. Symmetrically, each line of $E_2$ meets ten lines from $E_2$ and eight from $E_1$.

Very recently, Degtyarev, Itenberg and Sertoz proved that, up to projective equivalence, there is a unique smooth quartic with $64$ lines over $\CC$ \cite{DIS}. As a consequence, $64$ lines on a smooth quartic over $\CC$ are always a complete intersection, and always satisfy the combinatorial properties outlined above.

\end{example}

\begin{example}[Lines on Fermat surfaces] \label{ex:Fermat}
Let $\mathfrak{F}_d\subseteq \PP^3$ be the Fermat surface of degree $d$ given by the equation $x_0^d+x_1^d+x_2^d+x_3^d=0$. Let us denote by $\zeta_1, \zeta_2, \ldots, \zeta_d$ the $d$-th roots of unity. Let us also fix a primitive $2d$-th root of unity $\omega$. Adapting the notations in \cite{SSV}, let us list the lines of $\mathfrak{F}_d$ as follows (both $i$ and $j$ range from $1$ to $d$):
\[
\begin{array}{lll}
l_1(i,j) &=&\{[\lambda,\omega\zeta_i\lambda,\mu,\omega
\zeta_j\mu]\}, \\
l_2(i,j) &=&\{[\lambda,\mu,\omega\zeta_i\lambda,\omega
\zeta_j\mu]\},\\
l_3(i,j) &=&\{[\lambda,\mu,\omega
\zeta_j\mu,\omega\zeta_i\lambda]\}.
\end{array}
\]
This yields an arrangement $F$ of $3 \cdot d \cdot d = 3d^2$ lines. The dual graph $G_F$ consists of the following edges:
\begin{compactenum}
\item $\{l_a(i,j),l_a(h,k)\}$ iff $i=h$ or $j=k$, $a=1,2,3$;
\item $\{l_1(i,j),l_2(h,k)\}$ iff $i-j=h-k$ $\mod d$;
\item $\{l_1(i,j),l_3(h,k)\}$ iff $i+j=h-k$ $\mod d$;
\item $\{l_2(i,j),l_3(h,k)\}$ iff $i+j=h+k$ $\mod d $.
\end{compactenum}
We can partition the vertex set of $G_F$ into $3d$ copies of the complete graph $K_d$. In fact, if we fix $(a,i)\in \{1,2,3\}\times \{1,\ldots ,d\}$, the induced subgraph corresponding to the set of lines $F_{(a,i)}:=\{l_{a}(i,j),j=1,\ldots,d\}$ is complete. So each set of lines $F_{(a,i)}$ spans a plane $\pi_{a,i}$ in $\PP^3$. Hence, $F$ can be seen as the complete intersection of the surface $\mathfrak{F}_d$ with the union of the $3d$ planes $\pi_{a,i}$ just described. In particular,  one has  \[\reg X = d + 3d -1 = 4d-1.\]
One can see from the description above that $G_F$ is a $(4d-2)$-regular graph. This is consistent with Theorem \ref{thm:Main}: being $\mathfrak{F}_d$ smooth, the arrangement $F$ has only planar singularities.
Note also that the Fermat surface of degree $d=3$ is a smooth cubic. Since the dual graph does not depend on the chosen cubic, the dual graph of  ``27 lines on a smooth cubic'' is thus a particular case of the dual graph of ``$3d^2$ lines on the Fermat surface''.
\end{example}

\begin{example}[A 12-line arrangement different than Schl\"afli's]
For $d \ge 3$, let $\mathfrak{F}_d$ be the degree-$d$ Fermat surface and let $\pi_{a,i}$ be the $3d$ planes described in Example \ref{ex:Fermat}. For some integer $h \in \{1, \ldots, 3d\}$, choose $h$ of these $3d$ planes and let $\Pi_h$ be their union.
By the construction explained in Example \ref{ex:Fermat}, the intersection $F_h = \mathfrak{F}_d \cap \Pi_h$ is a subarrangement of $F$, consisting of exactly $hd$ lines. Clearly $F_h$ is a complete intersection  and the dual graph $G_{F_h}$ has regularity $h+d-2$. If we choose the degree-$3$ Fermat surface as smooth cubic, the ``Steiner set'' can be viewed as the case $h=3$, $d=3$ of this construction. Let us focus instead on the case $h=4$, $d=3$. This yields a $12$-line arrangements with $30$ intersection points that is  different than Schl\"{a}fli's double six. In fact, the graph $G_D$ of Example \ref{ex:Schlaefli} has diameter $3$ and it is bipartite, whereas $G_{F_4}$ has diameter $2$ and contains triangles. However, because of Theorem  \ref{thm:Main}, both graphs $G_D$ and $G_{F_4}$ are $5$-regular.
\end{example}

In many of the above examples, we considered the arrangement consisting of all the lines on some smooth surface in $\PP^3$, and noticed that the dual graph is regular. Unfortunately, this is not always the case, as shown by the examples in \cite{Ra}. However we would like to pose the following question:

\begin{question}
Let $d\geq 3$ and $X\subseteq \PP^3$ be a smooth surface of degree $d$ containing the maximum possible number of lines (among the smooth surfaces of degree $d$). Is it true that there is a natural number $N$ such that each of the lines meets exactly $N$ of the others?
\end{question}

As remarked in the above examples the answer is yes for $d=3$ and, over the complex numbers, for $d=4$. Furthermore by Theorem \ref{thm:Main}, the answer would be yes in general if the divisor of the lines was linearly equivalent to a multiple of the hyperplane section.

\section{The nerve complex}
Let $R$ be a standard graded ring with $R_0=\CC$. Let $\mm$ be its homogenous maximal ideal. Let $\pp_1, \ldots, \pp_s$ be the minimal primes of $R$. The \emph{nerve complex} or \emph{Lyubeznik complex} $\Lyu(R)$ is the simplicial complex on the vertices $1, \ldots, s$ described by 
\[\{i_1, \ldots, i_k\} \textrm{ is a face } \: \Longleftrightarrow \: \sqrt{\pp_{i_1} + \ldots + \pp_{i_k}} \neq \mm.\]  
Equivalently, if $X= \Proj R$ and $X_{1}, \ldots, X_{s}$ are its irreducible components, one could describe the nerve complex as follows:
\[\{i_1, \ldots, i_k\} \textrm{ is a face } \: \Longleftrightarrow \: X_{i_1} \cap \ldots \cap X_{i_k}  \textrm{  is nonempty}.\]  
It is a straightforward consequence of Borsuk's Nerve Lemma that any simplicial complex $\Delta$ is homotopy equivalent to $\Lyu(\CC[\Delta])$, where $\CC[\Delta]$ is the usual notation for the Stanley--Reisner ring of $\Delta$. However, $\Delta$ and $\Lyu(\CC[\Delta])$ are not homeomorphic in general: for example, when $\Delta$ is a simplex,  $\Lyu(\CC[\Delta])$ consists of a single point.

In 1962, Hartshorne \cite[Prop.~2.1]{Ha} showed that if $R$ is a standard graded ring and $\mathrm{depth} R \ge 2$, then $\Proj R$ is connected; or equivalently,  \[\widetilde{H}_0(\Lyu(R);\mathbb{C})=0.\]
Recently, Katzman, Lyubeznik and Zhang proved the following beautiful extension:

\begin{theorem}[\cite{KLZ}] If $R$ is a standard graded ring and $\mathrm{depth} R \ge 3$, then 
\[ \widetilde{H}_0(\Lyu(R);\mathbb{C})=\widetilde{H}_1(\Lyu(R);\mathbb{C})=0.\] 
\end{theorem}

It is not clear whether the obvious generalization to higher depth is true: the proof given in \cite{KLZ} relies on the fact that, because $\mathrm{depth} R\geq 3$, the cohomological dimension of the complement of $\Proj R$ in $\PP^n$ (embedded by the linear system $R_1$) is no more than $n-2$, as proven by the third author \cite{Va}. However, the latter fact is false for higher depths.

The theorem of \cite{KLZ} naturally suggests the question: what about the converse?

\begin{question}\label{q:1}
Given a simplicial complex $\Delta$ with vanishing $0$-th and $1$-st homology, is it the Lyubeznik complex of a ring of depth $\ge 3$?
\end{question}

Without any ``depth request'', the answer is known. Every simplicial complex $\Delta$ is the nerve of some ring: in fact, even of some Stanley--Reisner ring. Here is a simple construction illustrating this fact that was suggested to us  
by Alessio D'Al\`i.

\begin{lemma}\label{lem:ale}
Let $\Delta$ be a simplicial complex on $n$ vertices and dimension $d-1$. 
Let $N$ be the number of facets of $\Delta$. Let 
$M$ be  the maximum, taken over all vertices $v$ of $\Delta$, of the number of facets containing $v$. 
There is a simplicial complex $\Gamma$ on $s$ vertices, with $N \le s \le n +N$, and of dimension either $M-1$ or $M$, such that 
\[ \Lyu(\mathbb{C}[\Gamma]) = \Delta.\]
(Or in the words of combinatorialists: The nerve of $\Gamma$ coincides with $\Delta$.)
\end{lemma}

\begin{proof}
Let $1, \ldots, n$ be the vertices of $\Delta$, and let $F_1, \ldots, F_N$ be its facets. 
Let $A_1, \ldots, A_n$ be subsets of $[N]$ defined as follows: 
\[ A_i = \{ j \in [N] \textrm{ s.t. }  i \in F_j\}.\]
These $A_i$'s might not be facets of some simplicial complex, because a priori it can happen that some $A_i$ is contained in some other $A_{i'}$. This can be fixed as follows: if $A_i$ is contained in some $A_{i'}$, with $i \ne i'$, we update the set $A_i$ by adding to it the integer $N+i$. 
Let now $\Gamma$ be the simplicial complex generated by the $A_1, \ldots, A_n$. By construction, the intersection of $A_{i_1}, \ldots, A_{i_k}$ is non-empty if and only if $\{i_1, \ldots, i_k\}$ is contained in a facet of $\Delta$. 
\end{proof}

If we try to use the construction above to gather information on the depth of $\Gamma$, however, we are doomed. In fact, from $\Gamma$ one cannot even recover the dual graph of~$\Delta$. Moreover, the dimension of $\Gamma$ can be arbitrarily larger (or also smaller) than $\dim \Delta$.

To bypass these difficulties and tackle Question \ref{q:1}, we need to introduce a more geometric construction.

\begin{theorem} \label{thm:Lyubeznik}
Let $\Delta$ be a $(d-1)$-dimensional simplicial complex $d\geq 2$. There exists a $d$-dimensional standard graded $\mathbb{C}$-algebra $R=R(\Delta)$ such that:
\begin{compactenum}[\rm (i)]
\item $\Lyu(R)=\Delta$;
\item the dual graph of $R$ is the 1-skeleton of $\Delta$.
\end{compactenum}
Furthermore, if $\widetilde{H}_0(\Delta;\mathbb{C})=0$, then one can choose $R$ such that $\mathrm{depth}(R) \ge 2$; and if $\widetilde{H}_0(\Delta;\mathbb{C})=\widetilde{H}_1(\Delta;\mathbb{C})=0$, then one can choose $R$ so that $\mathrm{depth}(R)\geq 3$.
\end{theorem}
\begin{proof}
Let us fix a total order $\sigma_1, \ldots, \sigma_r$ of the minimal \textit{non}-faces of $\Delta$, so that 
\[\dim \sigma_i < \dim \sigma_j \Rightarrow i > j ,\]
that is, higher-dimensional non-faces are listed first.

Let us choose an arrangement of $n$ hyperplanes in $\PP^d$  (where $n$ is the number of vertices of $\Delta$) such that for each subset $A\subseteq [n]=\{1,\ldots ,n\}$
\[\dim \; \bigcap_{i \in A} H_i = \max\{-1,d - |A|\}.\]
This condition is obviously met if the hyperplanes are generic. 
Let now $X(\sigma_1)$ be the blow-up of $\PP^d$ along $\bigcap_{j \in \sigma_1} H_j$. For each $j\in[n]$, let  $H_j(\sigma_1)$ be the strict transform of $H_j$ in $X(\sigma_1)$. 
Recursively, for each $ i \in \{2, \ldots, r\}$ we denote
\begin{compactitem}
\item by $X(\sigma_{i})$ the blow-up of  $X(\sigma_{i-1})$  along $\bigcap_{j \in \sigma_{i}} H_j (\sigma_{i-1})$, and 
\item by $H_j(\sigma_{i})$ the strict transform of $H_j (\sigma_{i-1})$ in $X(\sigma_{i})$, for each $j\in[n]$.
\end{compactitem}
Consider $Y:=\bigcup_{j=1}^nH_j(\sigma_r)$. Being a blow-up of a projective scheme, by the blow-up lemma there exists a $d$-dimensional standard graded $\CC$-algebra $R$ such that $\Proj R=Y$. By construction, $\Lyu(R)=\Delta$. Furthermore, the dual graph of such an $R$ is the 1-skeleton of $\Delta$. This shows items (i) and (ii). Note also that the irreducible components of $Y$ and their nonempty intersections are smooth rational varieties. 

We are now left with the last two claims. First of all we observe  that $\Lyu(R)=\Delta$ is the nerve complex of $Y$ associated to the covering $\{H_j(\sigma_r)\}_{j=1,\ldots ,n}$. A smooth rational variety over $\mathbb{C}$ is simply connected (see \cite[Corollary 4.18]{debarre}), then by the extended version of the nerve lemma \cite[Theorem 6]{Bj}, we have
\begin{equation}\label{bjorner}
H^i(Y;\mathbb{C})\cong H^i(\Delta;\mathbb{C}) \ \ \mbox{ for } i=0,1.
\end{equation}
By the genericity of the hyperplanes $H_j\subseteq \PP^d$, $Y$ has only simple normal crossing singularities. So the natural maps
\[H^i(Y;\mathbb{C}) \: \longrightarrow \: H^i(Y;\O_Y)\]
are surjective for all $i$. Furthermore we have the Kodaira vanishing (e.g. \cite[9.12]{Ko}):
\[H^i(Y;\O_Y(k))=0 \ \ \forall \ k<0, \ i <d-1.\]
Denoting with $\mm$ the unique homogeneous maximal ideal of $R$, this translates into
\[H_{\mm}^1(R)_k=H_{\mm}^2(R)_k=0 \ \ \forall \ k< 0.\]
Obviously we can choose $R$ such that $H_{\mm}^0(R)=0$. Finally, by replacing $R$ with a high enough Veronese, we will also have that $H_{\mm}^1(R)_k=H_{\mm}^2(R)_k=0$ for all $k>0$. Since $\dim_{\CC}H_{\mm}^1(R)_0=\dim_{\CC} H^0(Y;\mathbb{C})-1$ and $\dim_{\CC}H_{\mm}^2(R)_0=\dim_{\CC} H^1(Y;\mathbb{C})$, the conclusion follows by \eqref{bjorner}.
\end{proof}

\noindent {\bf Acknowledgments}. 
We are indebted to Winfried Bruns, for a question to the third author that eventually inspired Theorem \ref{thm:Main}; to Bernd Sturmfels, for drawing our attention to the 27 lines; to Daniele Faenzi, for suggesting applications to $K3$ surfaces; to Alessio D'Al\`i, for Lemma \ref{lem:ale}. The second and third author wish also to thank the math department of the University of Miami for the warm hospitality in February 2016.



\begin{thebibliography}{BMS2}
\small
\itemsep=-1.5mm
%
\bibitem[BV2015]{BV}
B. Benedetti, M. Varbaro, {\it On the dual graph of Cohen-Macaulay algebras}, Int. Math. Res. Not. 2015, no. 17, pp. 8085-8115.

\bibitem[Bj2003]{Bj}
A. Bj\" orner, {\it Nerves, fibers and homotopy groups}, J. Comb. Theory A 102 (2003) 88--93.

\bibitem[BS2007]{BS}
S. Boissi\`ere, A. Sarti, {\it Counting lines on surfaces}, Ann. Sc. Norm. Super. Pisa Cl. Sci. (5) 6 (2007), no. 1, 39-52.

\bibitem[Ca1849]{Ca} A. Cayley, {\it On the triple tangent planes of surfaces of the third order}, Cambridge and Dublin Math. J. 4 (1849) 118--138.

\bibitem[Cl1861]{Cl} A. Clebsch, {\it Zur Theorie der algebraischer Fl\"achen}, J.~Reine~Angew.~Math.~58 (1861), 93--108.

\bibitem[De2001]{debarre}
O.~Debarre, {\it Higher-dimensional algebraic geometry}, Universitext, Springer-Verlag, New York, 2001.

\bibitem[DIS2016]{DIS}
A.~Degtyarev, I.~Itenberg, A.S.~Sertoz, {\it Lines on quartic surfaces}, Math. Ann. (2017) 368--753.

\bibitem[DS2002]{DS}
H. Derksen, J. Sidman, \emph{A sharp bound for the Castelnuovo-Mumford regularity of subspace arrangements}, Adv. Math. 172, pp. 151--157, 2002.

\bibitem[DuV1933]{DuVal}
P.~Du Val, {\em On the directrices of a set of points in a plane}. Proc. London Math. Soc. Ser. 2 35 (1933), 23--74.


\bibitem[GW1985]{GW}
A.~V.~Geramita and C.~A.~Weibel, {\em On the Cohen-Macaulay and Buchsbaum property for unions of planes in affine space}, J. Algebra 92 (1985), no. 2, 413--445

\bibitem[Mac2]{M2}
D.R.~Grayson, R.~Daniel, M.E.~Stillman, {\em Macaulay2, a software system for research in algebraic geometry}, Available at \url{http://www.math.uiuc.edu/Macaulay2/}.

\bibitem[Har1962]{Ha}
R. Hartshorne, {\it Complete intersection and connectedness}, Amer. J. Math. 84, pp. 497-508, 1962.

\bibitem[KLZ2015]{KLZ}
M.~Katzman, G.~Lyubeznik, and W.~Zhang, {\it An extension of a theorem by Hartshorne}, 
Proc. Amer. Math. Soc. 144 (2016), no. 3, 955--962. 

\bibitem[Ko1995]{Ko}
J.~Kollar, {\it Shafarevic Maps and Automorphic Forms}, Princeton University Press, 1995.


\bibitem[Ra2004]{Ra} S. Rams, \emph{Projective surfaces with many skew lines}, Proc. Amer. Math. Soc.133 (2004), no. 1, 11--13.

\bibitem[RS2015]{RS} S. Rams, M. Sch\"utt, \emph{64 lines on smooth quartic surfaces}, Math. Ann. 362 (2015), no. 1-2, 679--698. . 

\bibitem[Sa1974]{SD} B. Saint-Donat, {\it Projective models of K3 surfaces}, Amer. J. Math. 96 (1974) 602--639.

\bibitem[Sa1849]{Sa} G. Salmon, {\it On the triple tangent planes to a surface of the third order}, Cambridge and Dublin Math. J. 4 (1849) 252--260.

\bibitem[Sch1910]{Schoute}
P.~H.~Schoute,  {\em On the relation between the vertices of a definite six-dimensional polytope and the lines of a cubic surface}. Proc. Roy. Acad. Amsterdam 13 (1910), 375--383.

\bibitem[Sc1882]{Schur} F. Schur, {\it \"{U}ber eine besondere Classe von Fl\"achen vierter Ordnung}, Math. Ann. 20 (1882) 254--296.


\bibitem[SSV2010]{SSV} M. Sch\"utt, T. Shioda, and R. Van Luijk, {\it Lines on Fermat Surfaces}, J. Number Theory 130 (2010), no. 9, 1939-1963.

\bibitem[Se1943]{Se} B. Segre, {\it The maximum number of lines lying on a quartic surface}, Quart. J. Math., Oxford Ser. 14 (1943), 86-96. 

\bibitem[TT2015]{TT}
Z. Teitler, D.A. Torrance, {\it Castelnuovo-Mumford regularity and arithmetic Cohen-Macaulayness of complete bipartite subspace arrangements}, J. Pure Appl. Algebra 219 (2015), no. 6, 2134--2138. 

\bibitem[Va2013]{Va}
M. Varbaro, {\it Cohomological and projective dimensions}, Compos. Math. 149 (2013), 1203--1210.

\end{thebibliography}
\end{document}